\documentclass{amsart}

\usepackage{hyperref}

\usepackage{amssymb}
\usepackage{amsmath}
\usepackage{amsthm}
\usepackage{enumerate}
\usepackage{graphicx}
\usepackage{color}

\theoremstyle{plain}
\newtheorem{thm}{Theorem}[section]
\newtheorem{lem}[thm]{Lemma}
\newtheorem{prop}[thm]{Proposition}
\newtheorem{cor}[thm]{Corollary}

\newtheorem{de}[thm]{Definition}

\begin{document}
	
\title[The divergence of Mock Fourier series for spectral measures]
	{The divergence of Mock Fourier series for spectral measures}
	
\author{Wu-yi Pan and Wen-hui Ai$^{\mathbf{*}}$}

\address{Key Laboratory of Computing and Stochastic Mathematics (Ministry of Education), School of Mathematics and Statistics, Hunan Normal University, Changsha, Hunan 410081, P. R. China}
	
\email{pwyyyds@163.com}
\email{awhxyz123@163.com}	
	
%\date{\today}
	

\thanks{The research is supported in part by the NNSF of China (No. 11831007)}	

\thanks{$^{\mathbf{*}}$Corresponding author}
	
	
\begin{abstract}
In this paper, we study divergence properties of Fourier series on Cantor-type fractal measure, also called Mock Fourier series.  We give a sufficient condition under which the Mock Fourier series for doubling  spectral measure  is divergent on non-zero set. In particularly, there exists an example of the quarter Cantor measure whose Mock Fourier sums is not almost everywhere convergent.
\end{abstract}
	
\maketitle
\section{ Introduction }
	Let $\mu$ be a Borel probability measure on $\mathbb{R}^d$ with compact support. We say that $\mu$ is a spectral measure if there exists a discrete set  $\Lambda \subset \mathbb{R}^d$ such that $E(\Lambda):=\{e^{-2\pi i\lambda \cdot x}: \lambda \in \Lambda\}$ is an orthonormal basis for $L^2(\mu)$. The first singular, nonatomic, spectral measure was constructed by
	Jorgensen and Pedersen \cite{JP98}.  They proved the surprising result that the quarter Cantor measure  $\mu_4$ is  a spectral.  Over twenty years, many other interesting singular spectral measures on 
	 self-affine and Moran fractal sets have been constructed (see \cite{AH14, DHL14, DHL19} and so on).

 Given a spectral measure $\mu$ with spectrum $\Lambda$, for $L^1(\mu)$ function $f$, 
 we define coefficients $c_{\lambda}(f)=\int f(y) e^{-2 \pi i \lambda \cdot y} d \mu(y)$ and the Mock Fourier series $\sum_{\lambda \in \Lambda} c_{\lambda}(f) e^{2 \pi i \lambda \cdot x} .$ There is a natural sequence of finite subsets $\Lambda_{n}$ increasing to $\Lambda$ as $n \rightarrow \infty$, and we define the partial sums of the Mock Fourier series by
 $$
 S_{n}(f)(x)=\sum_{\lambda \in \Lambda_{n}} c_{\lambda}(f) e^{2 \pi i \lambda \cdot x}.
 $$
 We will use $(S_n,\Lambda_n)$ to denote \textit{the Mock Dirichlet summation operator $S_n$ with $\Lambda_n$}.

As analogue to classical Fourier analysis, an extremely natural question is whether
$S_n(f)$ converges to $f$ as $n \to \infty$. The answer has an added piquancy since:
not only does it depend on the determining what the function space $f$ is belonged to, but it also depends critically on how one defines ``convergence''.	
	
\indent  As we all know, for any continuous function, the assertion of uniform convergence of classical Fourier series is wrong \cite{Zy68}. By contrast,	Strichartz \cite{Str06} showed that it is true for some singular continuous spectral measure with standard spectrum. Unfortunately, for given spectral measure, different spectrum may have different convergence. Dutkay et al. \cite{DHS14} proved there is a continuous function $f$ whose $(S_n(f),\Lambda_n)$ does not even converge pointwise to $f$ for standard spectrum 
$\Lambda=\bigcup_{i=0}^{\infty}\Lambda_n$.

 In this paper, we study divergence properties of doubling spectral measures. Let $(M, \rho)$ be a metric space and suppose that $\mu$ is a positive locally finite Borel measure on $M$. We call $\mu$ a doubling measure if $\mu$ satisfies the doubling condition
\[
\mu(B(x,2r)) \leq A_1 \mu(B(x,r))<\infty \] 
for all $x\in M$ and $r>0$, where $A_1$ is constant and independent of $x,r$. Here $B(x,r)$ denotes the closed ball $B(x,r)=\{y\in M:\rho(y,x)\leq r\}$. 

To conveniently state our main results, we just briefly introduce the related concepts, and their well-posedness will be given in Section 2. Recall the finite discrete measure defined on a measure space $(X,\mathcal{A},\mu)$ has the form
$\nu = \frac{1}{N}\sum_{i=1}^N\delta_{x_i}$  for every finite collection $x_1,x_2,\ldots x_N\in X$ not necessarily pairwise different points, where $\delta_{x}$ is the Dirac measure concentrated at the point $x\in X$. Given a Mock Dirichlet summation operator $S_n$ with $\Lambda_n$, we formally write 
\[
	S_n(\nu)(x) = \frac{1}{N}\sum_{\lambda \in \Lambda_{n}}\sum_{n=1}^N  e^{2 \pi i \lambda \cdot (x-x_n)}.
\]
Now we state our main result. 

\begin{thm}\label{thm1.1}
 Assume $\mu$ is a doubling spectral  measure. Let $(S_n,\Lambda_n)$ be the Mock Dirichlet summation operator and   if
\[
\lim_{\alpha \to \infty}\sup_{\nu = \frac{1}{N}\sum_{i=1}^N\delta_{x_i}}\mu(\{x\in X:\sup_{n}\vert S_n (\nu)(x)\vert >\alpha\})>0,
\]
then there exists an integrable function such that the Mock Fourier  series diverges on $\mu$-non-zero set.
\end{thm}

As a corollary, we find the example in \cite{DHS14} diverges on non-zero set.
\begin{cor}\label{cor1.2}
Let $\mu_{4}$ be quarter Cantor measure and suppose $(S_n,17\Lambda_n)$ is the Mock Dirichlet summation operator with  
\[
17\Lambda_n =\{17\sum_{j=0}^{n} 4^j l_j: \; l_j \in \{0,1\},\; n\in \mathbb{N}\}.
\]
Then there exists an integrable function $f\in L^{1}(\mu_{4})$ whose Mock Fourier series $S_n(f)(x)$ diverges on a nonzero set.
\end{cor}

We organize our paper as follows. In Section 2, we firstly present a brief overview of the the relationship between the continuity of maximal operators and  convergence almost everywhere. Succeeded by, we introduce the main tool for our proof of Theorem \ref{thm1.1}, i.e., the dyadic cube analysis constructed by Christ\cite{Ch90}. In Section 3, as an application of Theorem \ref{thm1.1}, we consider the self-affine mesure generated by Hadamard triple. Under some technical condition about the spectrum, we give a criterion on there exists an integrable function whose Mock Fourier series diverges on non-zero set. The criterion can be applied to cover Corollary \ref{cor1.2}.

\section{The proof of theorem \ref{thm1.1}}\label{SE:2}

Let $(X, \mathcal{A}, \mu)$ be a complete finite  measure space with a  $\sigma$-field $\mathcal{A}$. To recall some basic facts firstly, the space of (equivalence classes of) all measurable functions on$(X, \mathcal{A}, \mu)$ is denoted by $L^{0}(\mu)$. It is endowed with the topology of convergence in measure by the metric
\[
d(f, g) =\int_{X}\frac{|f-g|}{1 + |f-g|}d\mu.
\]
It is not difficult to show that $(L^0(\mu), d)$ is a complete metric space.

\indent A mapping $T : (M, d_1) \longmapsto  (L^0(\mu),d)$ from a metric space $M$ to $L^0(\mu)$ is said to be continuous at $x\in M$, if for any sequence $\{x_n\}\subset M,$ $ n\geq 1$, we have $d(Tx_n, Tx) \longrightarrow 0$ whenever $d_1(x_n, x) \longrightarrow 0$.	We call mapping $T$ is continuous if it is continuous at every point of $M$. 

 We firstly recall the following theorem due to Guzman \cite{Gu81} on the almost everywhere finiteness of maximal operator.

\begin{thm}\cite[p. 10]{Gu81}\label{thm2.1}
Assume $(X, \mathcal{A}, \mu)$ is a complete measure space and $T_k: L^1(\mu) \longmapsto (L^0(\mu), d)$ is a  sequence  of sub-linear  operators with $\mu(X)< \infty$. If each $T_k$ is continuous and that  the  maximal operator $T^*$ defined for $f\in L^1(\mu)$ and $x\in X$ as 
\begin{equation*}
  T^{*}f(x)=\sup_{k}\vert T_k f(x)\vert <\infty \quad \mu- \text{a.e.}
\end{equation*}
 Then $T^*$ is also continuous
at $0$, and therefore
\begin{equation*}
\lim\limits_{\alpha \to \infty}\phi(\alpha ):=\lim\limits_{\alpha  \to \infty}\sup_{\Vert f\Vert_{L^1(\mu)}\leq 1}\mu(\{x\in X:T^{*}f(x)>\alpha \})=0.
\end{equation*}
\end{thm}

Notice that bounded linear operator is always continuous. By Theorem \ref{thm2.1}, if $\{T_k\}$  is a  sequence  of  bounded linear  operators, then $\lim\limits_{\alpha \to \infty}\phi(\alpha)>0$ implies there exists an integrable function $g$ such that $T_k(g)$ is not almost everywhere convergent.

\begin{cor}\label{thmx:Gu}
	Let $(X,\mathcal{A},\mu)$ be a complete measure space and $T_k: L^1(\mu) \longmapsto (L^0(\mu), d)$ be a  sequence  of bounded linear operators with $\mu(X)< \infty$. If
	\begin{equation}\label{eq:alpha}
			\lim\limits_{\alpha  \to \infty}\sup_{\Vert f\Vert_{L^1(\mu)}\leq 1}\mu(\{x\in X:T^{*}f(x)>\alpha \})>0,
		\end{equation}
	then a function $g$ existing in $L^1(\mu)$ can be obtained that its $T_k(g)$  diverges on $\mu$-non-zero set in $X$.
\end{cor}

However, it is difficult to verify \eqref{eq:alpha}, but we can consider  a sum of Dirac measure firstly, which is rather easy to handle in many cases. Concretely, let $(X, \mathcal{A}, \mu)$ and $(X, \mathcal{A}, \nu)$ be complete Borel measure spaces defined on a  Hausdorff space $X$, and $\mathcal{B}(X)$ be the space of locally  finite Borel space measure on $X$. Consider a sequence $k_j$ of kernels satisfying the following two properties:
\begin{enumerate}[\rm(i)]
		\item Each $k_j: X\times X\mapsto X$ is a measurable function such that $k_j(\cdot,y)\in L^1(\mu)$. 
		\item For each $j$ there exists $O_j$ such that \begin{equation*}
		\Vert k_j(\cdot,y)\Vert_{L^1(\mu)}\leq O_j \quad \text{for every} \;y\in X.
		\end{equation*}
\end{enumerate}
We write
\[
	K_jf(x) = \int_{X}k_j(x,y)f(y)d\mu(y)\quad \text{for} \; f\in X.
\]
Using Fubini-Tonelli's theorem, the second property of kernels makes the  maximal operator make sense as following:
\[
K^{*}f(x)=\sup_{j}\vert K_jf(x)\vert.
\]
Moreover, if $k_j(x,y)$ is a continuous function with compact support on $X\times X$ for any $j\in \mathbb{N}$, then each $K_j$ is bounded linear operator from $L^1(\mu)$ to $L^{\infty}(\mu)$. Such an operator has a natural extension to a bounded linear operator from $\mathcal{B}(X)$ to $L^{\infty}(\mu)$, which we denote by $K_j$ again, namely
\begin{equation}\label{2}
	K_{j}\nu(x)=\int_{X}k_j(x,y)d\nu(y), \quad K^{*}\nu(x)=\sup_{j}\vert K_j\nu(x)\vert.
\end{equation}
Especially, choose a sum of Dirac measure  $\nu=\frac{1}{H}\sum_{h=1}^H\delta_{x_h}$ for $x_1,\cdots,$ $x_H\in X$, then
\[
K_{j}\nu(x)=\frac{1}{H}\sum_{h=1}^{H}k_j(x,x_h), \quad K^{*}\nu(x)=\frac{1}{H}\sup_{j}\vert \sum_{h=1}^{H}k_j(x,x_h)\vert.
\]

In what follows, our aim is to extend  the "pointillist principle" of Carena \cite[Theorem 1]{Ca09}, then the similar conclusion can be obtained under slightly different condition. To certificate the Theorem \ref{thm1.1}, recall the dyadic cubes constructed by Christ in \cite{Ch90}, which is very important for extending results from harmonic analysis to the metric space setting. 

	\begin{thm}\cite[Theorem 11]{Ch90}\label{thm:DS}
 Let $(X, \rho)$ be a metric space and suppose that $\mu$ is a regular doubling measure on $X$. Then there exists  a collection of open subsets
		\[
		\{Q^k_\alpha \subset X:\, k\in \mathbb Z,\; \alpha\in I_k\}
		\] 
satisfying the following properties.
\begin{enumerate}
[\rm(i)]
\item For each integer $k$,
		\[\mu \left(X\backslash \bigcup \nolimits _{\alpha}Q^k_\alpha\right)=0.\]
\item Each  $Q^k_\alpha $ has a center $z_{Q^k_\alpha }$ such that
		\[B(z_{Q^k_\alpha },C_{{1}}\delta ^{{k}})\subseteq Q^k_\alpha \subseteq B(z_{Q^k_\alpha },C_{{2}}\delta ^{{k}}),\]
		where $C_1,C_2$ and $\delta$ are positive constants depending only on the doubling constant $A_1$ of the measure $\mu$ and independent of $Q^k_\alpha $.
\item For each $(k,\alpha)$ and each $l<k$, there is a unique $\beta$ such that $Q^k_\alpha \subset Q^l_\beta $.	
\item For any $k, \alpha$ and $t>0$, there exist constants $\delta \in (0,1),  C_3 < \infty, \eta > 0$ depending only on $\mu$ such that
	\[\mu \left\{x\in Q^k_\alpha :\; \rho(x, X\backslash Q^k_\alpha )\leq t \delta ^{k}\right\} \leq C_{3}t^{\eta }\mu (Q^k_\alpha ).\]
\end{enumerate}
\end{thm}
$I_k$ denotes some index set, depending on $k$. Dyadic cubes is constructed by  
\[ \bigcup_{k\in\mathbb{Z},\alpha\in I_k}\{Q^k_{\alpha}\}.\]
It should be noted that the center of  dyadic cubes satisfies maximal $\delta^k$-distance disperse condition, that is 
	\begin{equation} \label{eq:M}
		\rho(z_{Q^{k}_{\alpha}},z_{Q^{k}_{\beta}})\geq \delta^k \quad \text{for any} \;\alpha \neq \beta.
	\end{equation}
	In this context, maximality means that no new points of the space $X$ can be added
	to the set  $\{z_{Q^k_{\alpha}}\}$ such that \eqref{eq:M} remains valid. One other factor indeed, the last condition says that the area near the boundary of a "cube" $Q^{k}_{\alpha}$ is small.

The main result is as follow.
\begin{lem}\label{lem1}	
	Let $(X,\rho)$ be a metric space and let $\mu$ be a positive  regular Borel  measure
	satisfying the doubling condition on $X$. Let $\nu$ be a measure such
	that $d\nu =gd\mu$ with $g\in L^1_{loc}(X,\rho,\mu)$. Denote
	\begin{align*}
		\psi_\alpha(f):&= \nu(\{x\in X:|f(x)|>\alpha \})
	\end{align*}
	for a measurable function $f$ defined on $(X,\rho,\nu)$.
	If each kernel $k_j(x,y)$ is a continuous function with compact support on $X\times X$ and $K^{*}$ is defined in \eqref{2}, then
	\begin{equation}\label{6}
		\lim\limits_{\alpha \to \infty}\sup_{\Vert f\Vert_{L^1(\nu)}\leq 1}\psi_\alpha(K^{*}f)=0 \iff \lim\limits_{\alpha \to \infty}\sup_{\omega=\frac{1}{H}\sum_{h=1}^{H}\delta_{a_h}}\psi_{\alpha}(K^*\omega)=0
	\end{equation}
for every finite collection $a_1,a_2,\ldots a_H\in X$ not necessarily pairwise different points.
\end{lem}
\begin{proof}

For pairwise different points $a_1,a_2,\ldots a_H\in X$, we firstly denote following sets:
\begin{align*}
		\mathcal{M_{\mathbb{N}}}&= \{\omega=\frac{\sum_{h=1}^{H}c_h\delta_{a_h}}{\sum_{h=1}^{H}c_h}:\; c_h \in \mathbb{N}^+\}, \\
\mathcal{M_{\mathbb{Q}}}&= \{\omega=\frac{\sum_{h=1}^{H}c_h\delta_{a_h}}{\sum_{h=1}^{H}c_h}:\; c_h \in \mathbb{Q}^+\}, \\
\mathcal{M_{\mathbb{R}}}&= \{\omega=\frac{\sum_{h=1}^{H}c_h\delta_{a_h}}{\sum_{h=1}^{H}c_h}:\; c_h \in \mathbb{R}^+\},
\end{align*}
and
\begin{align*}
\mathcal{F}_{\mathcal{R}}&= \{f=\frac{\sum_{h=1}^{H}c_h\chi_{Q_h}}{\sum_{i=1}^{H}c_h\nu(Q_h)}:\; c_h \in \mathbb{R}^+, \; Q_i \cap Q_j=\emptyset\},
\end{align*}
where $Q_h$ is dyadic cube constructed in Theorem \ref{thm:DS} and $\chi_{Q_h}$ is characteristic function.

	{\bf Necessity. }We divide the proof of  
   \begin{equation*}
	\lim\limits_{\alpha \to \infty}\sup_{\omega=\frac{1}{H}\sum_{h=1}^{H}\delta_{a_h}}\psi_{\alpha}(K^*\omega)=0
\Longrightarrow	
\lim\limits_{\alpha \to \infty}\sup_{\Vert f\Vert_{L^1(\nu)}\leq 1}\psi_\alpha(K^{*}f)=0 
	\end{equation*}	
into four steps.

	{\bf Step 1. }By $\lim\limits_{\alpha \to \infty}\sup\limits_{\omega=\frac{1}{H}\sum_{h=1}^{H}\delta_{a_h}}\psi_{\alpha}(K^*\omega)=0$, it is easy to see
	$
	\lim\limits_{\alpha \to \infty}\sup\limits_{\omega\in \mathcal{M}_{\mathbb{N}}}\psi_\alpha(K^*\omega)
	$ $=0.$ In this step, we verify that
	$
	\lim\limits_{\alpha \to \infty}\sup\limits_{\omega\in \mathcal{M}_{\mathbb{Q}}}\psi_\alpha(K^*\omega)=0.
	$
	
	To prove this, we write $c_h=\frac{n_h}{m_h}$ with $n_h,m_h\in \mathbb{N}^+$ for $c_h\in \mathbb{Q}^+$. Since 
	\[
	\frac{\sum_{h=1}^{H}c_h\delta_{a_h}}{\sum_{h=1}^{H}c_h}	= \frac{\sum_{h=1}^{H}\overline{c}_h\delta_{a_h}}{\sum_{h=1}^{H}\overline{c}_h},\quad\text{where} \;\; \overline{c}_h=n_h\prod\limits_{j=1,j\neq h}^{H}m_j\in \mathbb{N},
	\]
we know $\mathcal{M}_{\mathbb{N}}=\mathcal{M}_{\mathbb{Q}}$.
Hence $
	\lim\limits_{\alpha \to \infty}\sup\limits_{\omega\in \mathcal{M}_{\mathbb{Q}}}\psi_\alpha(K^*\omega)=0
	$.
	
	{\bf Step 2. }In this step, we want to prove that  $
	\lim\limits_{\alpha \to \infty}\sup\limits_{\omega\in \mathcal{M}_{\mathbb{R}}}\psi_\alpha(K^*\omega)=0
	$. Let us start by defining maximal truncated operator $K^*_Nf(x)=\max\limits_{1\leq j\leq N}|K_jf(x)|$, then it is clear that 
	\begin{equation*}
		\begin{aligned}
			\lim\limits_{N\to \infty}\psi_\alpha(K^*_Nf) = 
			\psi_\alpha(K^*f)
		\end{aligned}
	\end{equation*} 
for each $\alpha >0$. Similarly, we define $K^*_N \omega$ for finite discrete measure $\omega$.
	
	We next claim that,  for each integer $N$, real numbers $\alpha,\varepsilon>0$, $0<\beta<\alpha$ and $\omega\in \mathcal{M}_{\mathbb{R}}$, we can find a finite discrete measure $\overline{\omega}\in \mathcal{M}_{\mathbb{Q}}$ satisfying the inequality 
	\begin{equation*}
		\psi_\alpha(K^*_N\omega)\leq \psi_{\alpha-\beta}(K^*_N\overline{\omega})+2\varepsilon.
	\end{equation*}
	
	In fact, for $\omega \in \mathcal{M}_{\mathbb{R}}$, take
	$d_h \in \mathbb{Q}^+$ such that $  c_h  =  d_h  +  r_h$, where $r_h>0$ will be determined later. Since 
	\[\frac{\sum\limits_{h=1}^{H}c_hk_{j}(x,a_h)}{\sum\limits_{h=1}^{H}c_h}
=\frac{\sum\limits_{h=1}^{H}d_hk_j(x,a_h)}{\sum\limits_{h=1}^{H}d_h}
+\frac{\sum\limits_{h=1}^{H}r_hk_j(x,a_h)}{\sum\limits_{h=1}^{H}d_h}
-\frac{\sum\limits_{h=1}^{H}c_hk_j(x,a_h)}{\sum\limits_{h=1}^{H}d_h}\cdot\frac{\sum\limits_{h=1}^{H}r_h}{\sum\limits_{h=1}^{H}c_h},
	\] 
we simplify above equality as
$$
K_j(c)=K_j(d)+K_j(rd)-K_j(crd).
$$
Then, for $0\leq\beta\leq \alpha$, we have 
	\[
	\begin{aligned}
&\psi_\alpha(K^*_N\omega)
=\psi_\alpha\left(\sup_{1\leq j\leq N} |K_j(c)|\right)\\
\leq & \psi_{\alpha-\beta}\left(\sup_{1\leq j\leq N} |K_j(d)|\right) 
		+\psi_{\frac{\beta}{2}}\left(\sup_{1\leq j\leq N} |K_j(rd)|\right)  	\\
&+\psi_{\frac{\beta}{2}}\left(\sup_{1\leq j\leq N} |K_j(crd)|\right).
	\end{aligned}
	\]
Furthermore,
	\[
	\begin{aligned}
\psi_{\frac{\beta}{2}}\left(\sup_{1\leq j\leq N} |K_j(rd)|\right)   
\leq & \psi_{\frac{\beta}{2}}\left(\frac{\sum_{j=1}^{N}\sum_{h=1}^{H}r_h|k_j(x,a_h)|}{\sum_{h=1}^{H}d_h}\right)\\ \leq & \frac{2}{\beta}\sum_{j=1}^{N}\sum_{h=1}^{H}|r_h|\int_{X}|k_j(x,a_h)|d\nu\cdot\frac{1}{\sum_{h=1}^{H}d_h}, \\
\psi_{\frac{\beta}{2}}\left(\sup_{1\leq j\leq N} |K_j(crd)|\right) \leq &\frac{2}{\beta}\sum_{j=1}^{N}\sum_{h=1}^{H}|c_h|\int_{X}|k_j(x,a_h)|d\nu\cdot\frac{\sum_{h=1}^{H}r_h}{\sum_{h=1}^{H}d_h\sum_{h=1}^{H}c_h},
	\end{aligned}
	\]
and 
$\int_{X}|k_j(x,a_h)|d\nu< \infty$
for all $j$, we can choose small $r_h$ such that the  right side hand of the above inequalities are all less than arbitrary $\varepsilon >0$. If we write $\overline{\omega}=\frac{\sum_{h=1}^{H}d_h\delta_{d_h}}{\sum_{h=1}^{H}d_h}$, our claim is proved. 
	
	Hence we use $K^*_N\leq K^*$ to get $\psi_\alpha(K^*_N\omega)\leq \sup\limits_{\overline{\omega}\in \mathcal{M_{\mathbb{Q}}}}\psi_{\alpha-\beta}(K^*\overline{\omega})+2\varepsilon.$
	Moreover, taking the maximum in the measure family $\mathcal{M_{\mathbb{R}}}$ and letting  $\beta \to 0$, $\varepsilon\to 0$, we have
	\[
	\sup\limits_{\overline{\omega}\in \mathcal{M_{\mathbb{Q}}}}\psi_{\alpha}(K^*\overline{\omega})\leq \sup\limits_{\omega\in \mathcal{M_{\mathbb{R}}}}\psi_\alpha(K^*_N\omega) \leq \lim_{\alpha_0\to \alpha^-}\sup\limits_{\overline{\omega}\in \mathcal{M_{\mathbb{Q}}}}\psi_{\alpha_0}(K^*\overline{\omega}).
	\]
	Observing that $\sup\limits_{\overline{\omega}\in \mathcal{M_{\mathbb{Q}}}}\psi_{\alpha}(K^*\overline{\omega})$ monotonically decreases with $\alpha$ and 
\[\lim\limits_{\alpha\to \infty}\sup\limits_{\overline{\omega}\in \mathcal{M_{\mathbb{Q}}}}\psi_{\alpha}(K^*\overline{\omega})=0.\]
By sandwich theorem, we obtain \[\lim\limits_{\alpha\to \infty}\sup\limits_{\overline{\omega}\in \mathcal{M_{\mathbb{R}}}}\psi_{\alpha}(K^*\overline{\omega})=0.\] 
By the way, this trick will be used repeatedly. To facilitate the writing, we will not describe this process in detail.
	
	{\bf Step 3. }In this step, we show that$
	\lim\limits_{\alpha \to \infty}\sup\limits_{f\in \mathcal{F}_{\mathbb{R}}}\psi_\alpha(K^*f)=0
	$.
	Repeat the trick in {\bf Step 2}, it is enough to prove that for each integer $N$, real numbers $\alpha,\varepsilon>0$, $0<\beta<\alpha$ and $f\in \mathcal{F}_{\mathbb{R}}$, there exists a finite discrete measure $\omega\in\mathcal{M}_{\mathbb{R}}$ satisfying the inequality 
	\[
	\psi_\alpha(K^*_Nf)\leq \psi_{\alpha-\beta}(K^*_N\omega)+\varepsilon.
	\]
	The desired limiting behavior will follow by letting  $\beta \to 0$, $\varepsilon\to 0$ and $N \to \infty$ and taking the maximum as before. 
	
	Let $f=\frac{\sum_{h=1}^{H}c_h\chi_{Q_h}}{\sum_{h=1}^{H}c_h\nu(Q_h)}\in \mathcal{F}_{\mathbb{R}}$ and
	$
	\omega= \frac{\sum_{h=1}^{H}c_h\nu(Q_h)\delta_{z_{Q_h}}}{\sum_{h=1}^{H}c_h\nu(Q_h)}\in \mathcal{M_{\mathbb{R}}},
	$
where $z_{Q_h}$ denotes the center of the dyadic cube $Q_{h}$.
	If $0\leq \beta \leq \alpha$, for fixed $N$, we obtain that
	\[
\psi_\alpha(K^*_Nf) \leq\psi_{\alpha-\beta}(K^*_N\omega) +\psi_\beta(K^*_N(f-\omega)). 
     \]
From 
   \[
	\begin{aligned}
		\left|\left(\sum_{h=1}^Hc_h\nu(Q_h)\right)K_j(f-\omega)\right|=&\left|\sum_{h=1}^Hc_h\left(\int_{Q_h}k_j(x,y)d\nu(y)-\int_{Q_h}k_j(x,z_{Q_h})d\nu(y)\right)\right|	 \\ \leq&\sum_{h=1}^Hc_h\int_{Q_h}|k_j(x,y)-k_j(x,z_{Q_h})|d\nu(y),
	\end{aligned}
	\]
we have
	\[
	\begin{aligned}
		&\psi_{\beta}(K^*_N(f-\omega))\\
		\leq&\sum_{j=1}^N\frac{1}{\beta}\int_{X}|K_j(f-\omega)(x)|d\nu(x)	 \\ \leq&\frac{1}{\beta\left(\sum\limits_{h=1}^Hc_h\nu(Q_h)\right)}\sum_{j=1}^N\int_X\left(\sum_{h=1}^Hc_h\int_{Q_h}|k_j(x,y)-k_j(x,z_{Q_h})|d\nu(y)\right)d\nu(x)\\
		\leq&\frac{1}{\beta\left(\sum\limits_{h=1}^Hc_h\nu(Q_h)\right)}\sum_{j=1}^N\sum_{h=1}^Hc_h\int_{Q_h}\left(\int_{F_j}|k_j(x,y)-k_j(x,z_{Q_h})|d\nu(x)\right)d\nu(y),
	\end{aligned}
	\]
where $F_j$ denotes the projection of the support of $k_{j}(x,y)$. Clearly, $F_j$ is a bounded set  with finite measure. Since every $k_j(x,y)$ is a uniformly continuous  function with compact support in $X \times X$, we can take small ${\rm diam} (Q_{i})$ such that the term in above inequalities is small enough. 
	Similar to the previous proof in {\bf Step 2}, we conclude that
	$
	\lim\limits_{\alpha \to \infty}\sup\limits_{f\in \mathcal{F}_{\mathbb{R}}}\psi_\alpha(K^*f)=0
	$.
	
	{\bf Step 4. }Finally, from the fact that the set  of all real coefficients linear combinations of  characteristic functions of dyadic sets is dense in $L^1(\mu)$, by a standard argument, one obtains 
	\[
	\lim\limits_{\alpha \to \infty}\sup\limits_{\Vert f\Vert_{1}\leq 1  }\psi_\alpha(K^*f)=\lim\limits_{\alpha \to \infty}\sup\limits_{f\in \mathcal{F}_{\mathbb{R}}}\psi_\alpha(K^*f)=0.
	\]
	The proof of Lemma  \ref{lem1} in one direction is complete.

	{\bf Sufficiency. }Conversely, we want to prove 
	\[
	\lim_{\alpha \to \infty}\sup_{\Vert f\Vert_{1}\leq 1}\psi_\alpha(K^*f)= 0
\Longrightarrow	
	\lim_{\alpha \to \infty}\sup_{\omega\in \mathcal{M}_{\mathbb{N}}}\psi_\alpha(K^*\omega)=0.
	\]
Clearly $\omega\in \mathcal{M}_{\mathbb{N}}\subset \mathcal{M_{\mathbb{Q}}}$. 
For pairwise different points $a_1,a_2,\ldots a_H\in X$, we denote the metric by $\rho$ and $d=\min \{\rho(x_i,x_j), x_i \neq x_j \}$.
Now we apply the properties (i) and (ii) in Theorem \ref{thm:DS}, then there exists $Q^n_{i_h}$ such that $x_h\in \overline{Q^n_{i_h}} $. Let $n$ be a large integer such that $C_2\delta^n<\frac{d}{4}$, where $C_2,\delta$ are constants mentioned in Theorem \ref{thm:DS}. We claim the set in $\{Q^n_{i_h}\}_{h=1}^{H}$ are pairwise disjoint. In fact, if $x\in Q^n_{i_h} \cap Q^n_{i_m} $ for $h\neq m$. Since $Q^n_{i_h}\subset B(z_{Q^n_{i_h}},C_2\delta^n)\subset B(z_{Q^n_{i_h}},\frac{d}{4})$ and $Q^n_{i_m}\subset B(z_{Q^n_{i_m}},C_2\delta^n)\subset B(z_{Q^n_{i_m}},\frac{d}{4})$, using the triangle inequality, we have
	\[
	\rho(x_h,x_m)\leq \rho(x_h,z_{Q^n_{i_h}})+\rho(z_{Q^n_{i_h}},x)+\rho(x,z_{Q^n_{i_m}})+\rho(z_{Q^n_{i_m}},x_m)<d.
	\]
	This is a contradiction to $\rho(x_h,x_m)\geq d$.	
	Apparently, let $c_h \in \mathbb N^+$ and
	\[
	\omega=\sum_{h=1}^H c_h\delta_{x_h} ,\quad  f=\frac{1}{\sum\limits_{h=1}^Hc_h}\sum_{h=1}^H\frac{c_h}{\nu(Q^n_{i_h})}\chi_{Q^n_{i_h}}, 
	\]
then $\omega\in \mathcal{M_{\mathbb{N}}}$ and $\Vert f\Vert_{1}=1$. Fix $N\in\mathbb{N}^+$ and $\beta>0$, through simple calculation, we have
	\[
	\begin{aligned}
		\psi_{\beta}(K^*_N(f-\omega))
		\leq &  \sum_{j=1}^N\psi_{\beta}(K_j(f-\omega))\leq\sum_{j=1}^N\frac{1}{\beta}\left|\int_{X}k_j(x,y)(f(y)d\nu-d\omega)\right|\\
	\end{aligned}
	\]
	and
	\[
	\bigg|\int_{X}k_j(x,y)(f(y)d\nu-d\omega)\bigg|	\leq\frac{1}{\sum\limits_{h=1}^{H}c_h}\sum_{h=1}^{H}\frac{c_h}{\nu\left(Q^n_{i_h}\right)}\left|\int_{Q^n_{i_h}}k_j(x,y)-k_j(x,x_h)d\nu(y)\right|.
	\]
Hence,
	\[
	\begin{aligned}
		&\psi_{\beta}(K^*_N(f-\omega))\\
		\leq& \frac{1}{\sum\limits_{h=1}^{H}\beta c_h} \sum_{h=1}^H \frac{c_h}{\nu\left(Q^n_{i_h}\right)}\sum_{j=1}^N\int_X\left(\int_{Q^n_{i_h}}|k_j(x,y)-k_j(x,x_h)|d\nu(y)\right)d\nu(x)\\
		\leq&\frac{1}{\sum\limits_{h=1}^{H}\beta c_h} \sum_{h=1}^H \frac{c_h}{\nu\left(Q^n_{i_h}\right)}\sum_{j=1}^N\int_{F_j}\left(\int_{Q^n_{i_h}}|k_j(x,y)-k_j(x,x_h)|d\nu(x)\right)d\nu(y),
	\end{aligned}
	\]
where $F_j$ is the projection of the support of $k_{j}(x,y)$. Similar to  {\bf Step 3}, we can choose the diameter of the dyadic sets such that $\psi_{\beta}(K^*_N(f-\omega))\leq \varepsilon$ for any $\varepsilon>0$. Repeat the same step like before, we obtain that $	\lim_{\alpha \to \infty}\sup_{\omega\in \mathcal{M}_{\mathbb{N}}}\psi_\alpha(K^*\omega)=0$. This completes the proof of Lemma \ref{lem1}.
\end{proof}
	
	Combining with Corollary \ref{thmx:Gu} and Lemma \ref{lem1}, we obviously have the following corollary.

\begin{cor}\label{co:1314}
	Let $(X,\rho)$ be a metric space.  Assume every kernel $k_j(x,y)$ is a continuous function with compact support on $X\times X$ and $K^{*}$ is defined in \eqref{2}.  If $\mu$ is a totally finite complete measure satisfing the doubling condition and
	\begin{equation*}
		\lim\limits_{\alpha  \to \infty}\sup_{\omega=\frac{1}{H}\sum_{h=1}^{H}\delta_{a_h}}\mu(\{x\in X:K^{*}\omega(x)>\alpha \})>0,
	\end{equation*}
	then there exists a function $g \in L^1(\mu)$ such that $T_k(g)$  diverges on $\mu$-non-zero set in $X$. 
\end{cor}
	Based on above lemmas, we can prove Theorem \ref{thm1.1}.

{\bf Proof of Theorem \ref{thm1.1}.} 
Let $\mu$ denote a spectral measure supported on a compact subset of $X\subset(\mathbb{R}^d,\rho)$, where $(\mathbb{R}^d,\rho)$ is the Euclidean space. Let $\{e^{-2\pi i\lambda \cdot x}: \lambda \in \Lambda\}$ be an exponential  orthonormal basis of $L^2(\mu)$. 
Define the Mock Dirichlet kernel as
\[
k_{n}(x,y)=\sum_{\lambda\in \Lambda_{n}}e^{2\pi i\lambda \cdot (x-y)}.
\] 
Then the Mock Dirichlet summation operator $S_n$ with $\Lambda_n$ can be written as 
\[ 
S_{n}(f)(x)=\sum_{\lambda \in \Lambda_{n}} c_{\lambda}(f) e^{2 \pi i \lambda \cdot x}=\int_Xk_n(x,y)f(y)d\mu(y).
\]
 From Corollary \ref{co:1314}, the theorem follows immediately.$\hfill\qed$ 

\section{Application to self-affine spectral measures}

In this section, we apply our results to self-affine spectral measures. 
Recall that the self-affine measure is defined by  iterated function system (IFS).

\begin{de}[Self-affine measure]\label{de:2}
Let $R$ be a $d\times d$ expansive matrix (all its
eigenvalues have modulus strictly bigger than one). Let $B=\{b_0,b_1,\cdots,b_{N-1}\}$ be a finite subset of $\mathbb{R}^{d}$. We define the affine iterated function system
\[
	\varphi_{b}(x)=R^{-1}(x+b)\quad \text{for}\; x\in \mathbb{R}^d \;\text{and}\; b\in B.
\]
The self-affine measure (with equal weights) is the unique probability measure  satisfying
\begin{equation*}\label{7}
			\mu(E) =\frac{1}{N}\sum_{j=1}^{N}\mu(\varphi_{b}^{-j}(E))\quad \text{for all Borel subsets}\; E \;\text{of}\;\;\mathbb{R}^d.
\end{equation*}
This measure is supported on the attractor $T(R,B)$ which is the unique compact
set that satisfies
\[
	T(R,B) = \bigcup_{b\in B}\varphi_{b}(T(R,B)).
\]
The set $T(R,B)$ is also called the self-affine set associated with the IFS. It can also
be described as
\[
	T(R,B) =\Big\{\sum_{k=1}^{\infty}R^{-k}b_k: \;  b_k \in B \Big\}.
\]
One can refer to  \cite{Hu81} for a detailed exposition of the theory of iterated
function systems. In this section, we will use $\mu_{R,\{-1,1\}}$ to denote Cantor measure which is the special case when $d=1$ and $B=\{-\frac{1}{4},\frac{1}{4}\}$.
\end{de}

To the best of our knowledge, most of self-affine spectral measures are constructed by Hadamard triples.
 
\begin{de}[Hadamard triple]
	For a given expansive $d\times d$ matrix $R$ with integer entries. Let $B,L \subset \mathbb Z^d$ be finite sets of integer vectors with the same cardinality $N\geq 2$. We say that the triple (R,B,L) forms a Hadamard triple if the matrix
	\[
	H= \frac{1}{\sqrt{N}}\left[e^{2\pi iR^{-1}b\cdot l}\right]_{l\in L, b\in B}
	\]
is unitary, i.e., $H^{*}H=I$, where $H^{*}$ denotes conjugate transpose of $H$. 
\end{de}
The system $(R,B,L)$ forms a Hadamard triple if and only if the Dirac measure $\delta_{R^{-1}D}=\frac{1}{\# B}\sum_{b\in B}\delta_{R^{-1}b}$ is a spectral measure with spectrum $L$. Moreover, this property is a key property in producing spectrum of self-affine spectral  measures.
{\L}aba and Wang \cite{LW02},  Dutkay,  Haussermann and Lai \cite{DHL19} eventually proved that Hadamard triple generates self-affine spectral measure in all dimension.

If $(R,B,L)$ forms a Hadamard triple, 
we let
	\begin{equation*}\label{eq:sp}
		\Lambda_n =L+R^{t}L+(R^t)^2L+\cdots+(R^t)^{n-1}L
      =\sum_{k=0}^{n-1}(R^{t})^{k}L,
	\end{equation*}
and
	\begin{equation}\label{eq:sp2}
		\Lambda =\bigcup_{n=0}^{\infty}\Lambda_n=\sum_{k=0}^{\infty}(R^{t})^{k}L,
	\end{equation}
where $R^{t}$ denotes transpose of $R$. 

The set $\Lambda$ forms an orthonormal set  for the self-affine spectral measure $\mu:=\mu_{R,B}$. But the set $\Lambda$ can be incomplete (see \cite[p. 4]{DLW17}).
In this paper, we assume that the self-affine spectral measure $\mu$ generated by Hadamard triple $(R,B,L)$ always has a spectrum like  \,(\ref{eq:sp2}).

We say that the self-affine measure $\mu$ in Definition \ref{de:2} satisfies the no-overlap condition or measure disjoint condition if
 \[
 	\mu(\varphi_{b}(T(R,B))\cap \varphi_{b^{'}}T(R,B)) = 0
 	\; \text{for all}\; b\neq b^{'} \in B. 
 \] 
Dutkay,  Haussermann and Lai \cite{DHL19} proved that if the self-affine spectral measure is generated by 
Hadamard triple, then the no-overlap condition is satisfied.

Following the work of Dutkay et al.\cite{DHS14}, encoding map  plays the key role in linking no-overlap self-affine spectral measure and code space.	
Let $\mathbb{N}^{*} $ denote the positive integer numbers. For symbolic space $B^{\mathbb{N}^*}$ with the product probability measure $dP$ where each digit in $B$ has probability $1/N$, Dutkay et al.\cite{DHS14} proved following proposition.
\begin{prop}\cite[Proposition 1.11]{DHS14}\label{prop:iso}
Define the encoding map $h: B^{\mathbb{N}^*}\mapsto T(R,B)$ by
		\[
			h(b_1b_2b_3...)=\sum_{i=1}^{\infty}R^{-i}b_i,
		\]
then $h$ is onto, it is one to one on a set of full measure, and it is measure preserving. 

Define the right shift $S:B^{\mathbb{N}^*}\mapsto B^{\mathbb{N}^*}$,
				\[
					S(b_1b_2b_3\cdots)=b_2b_3\cdots,
				\]
and the map $\mathcal{R}: T(R,B)\mapsto T(R,B)$,
		\[
		\mathcal{R}\left(\sum_{i=1}^{\infty}R^{-i}b_{i}\right)=\sum_{i=1}^{\infty}R^{-i}b_{i+1}.
		\]
Then 
\[hS=\mathcal{R}h.\]
\end{prop}

By Proposition \ref{prop:iso}, we immediately obtain the following corollary.
\begin{cor}\label{co:er}
	The dynamical systems of $(T(R,B),\mathcal{F},\mu,\mathcal{R})$ is ergodic.
\end{cor}
		
For a self-affine spectral measure $\mu:=\mu_{R,B}$ generated by Hadamard triple $(R,B,L)$,
let $\tau$ be an integer such that 
\[
\tau\Lambda =\tau\bigcup_{n=0}^{\infty}\Lambda_n=\tau\sum_{k=0}^{\infty}(R^{t})^{k}L
\]
is a spectrum of $\mu$.
 Define the Dirichlet kernel 
\[
	D_{n}(x) := \sum_{\lambda\in \tau\Lambda_n}e^{2\pi i \lambda \cdot x}
\;\;(x\in \mathbb R).
\]
For $f\in L^1(\mu)$, the Mock Dirichlet summation operator
\[
	 S_{n}(f)(x)=\sum_{\lambda \in \tau \Lambda_{n}} \left( \int_{T(R,B)} f(y) e^{-2 \pi i \lambda \cdot y}d\mu(y)\right)e^{2 \pi i \lambda \cdot x}
\]
can be written as 
\[
	S_{n}(f)(x)=\int_{T(R,B)} f(y) D_n(x-y)d\mu(y).
\]

Similar to \cite[Proposition 2.2]{DHS14}, we have following result. 

\begin{prop}\label{pr:D}
Define trigonometric polynomials
	\[
	m_{\tau}(x)=\sum_{l\in L}e^{2\pi i (\tau l)\cdot x}\;\;(x\in \mathbb R).
	\]
	Then the Dirichlet kernel satisfies the formula
	\begin{equation}\label{eq:ke}
			D_n(x)=\prod_{k=0}^{n}m_\tau((R^{t})^{k}x).
	\end{equation}

\end{prop} 	
\begin{proof}
   The proof is by induction on $n$.
	It is clear that 
	\begin{equation*}
		D_0(x)= \sum_{\lambda\in \tau\Lambda_0}e^{2\pi i \lambda \cdot x}=\sum_{l\in L}e^{2\pi i (\tau l) \cdot x}=m_\tau(x),
	\end{equation*}
Suppose $D_n(x)=\prod_{k=0}^{n}m_\tau((R^{t})^{k}x)$,
we need to prove 
	\begin{equation}\label{eq:ke2}
		D_{n+1}(x)=m_\tau(x)D_n(R^{t}x).
	\end{equation}
	Since $\Lambda_{n+1}=R^{t}\Lambda_{n}+L$, we see that every point $\lambda_{n+1}$ in $\Lambda_{n+1}$ will have a unique
	representation of the form $\lambda_{n+1}=R^{t}\lambda_n+l$ with $\lambda_n\in \Lambda_{n}$ and $l\in L$. This yields
	\begin{align*}
	D_{n+1}(x) &= \sum_{\lambda_{n}\in \Lambda_{n}}\sum_{l\in L}e^{2\pi     i\tau(R^t\lambda_{n}+l)\cdot x} \\
   &= \sum_{l\in L}e^{2\pi i (\tau l) \cdot x}\sum_{\lambda_{n}\in \Lambda_{n}}e^{2\pi  i\tau (R^t\lambda_{n})\cdot x}\\
  &=m_\tau(x)D_n(R^t x).
	\end{align*}
	Then equation (\ref{eq:ke}) follows by induction from equation (\ref{eq:ke2}).
\end{proof}

Using above propositions and Theorem \ref{thm1.1}, we can obtain following lemma about doubling spectral measure generated by Hadamard triple.
\begin{lem}\label{Classification}
Let $R$ be an integer symmetric matrix and let $\mu:=\mu_{R,B}$ be a self-affine spectral measure generated by Hadamard triple $(R,B,L)$. Assume $\mu$ is a doubling mesure with spectrum $\tau\Lambda =\sum_{k=0}^{\infty}R^{k}\tau L.$ Let
		\[
		\Delta(m_{\tau,b}):=\exp\bigg(\int_{T(R,B)} \log |m_{\tau}(x-(I-R^{-1})^{-1}b)|d\mu(y)\bigg),\; b\in B,
		\]
where $m_{\tau}(x)$ is defined in Proposition \ref{pr:D}.
		If $\Delta(m_{\tau,b})>1$ for some $b\in B$, then there exists an integrable function such that the Mock Fourier series diverges on non-zero set.
	\end{lem}
	\begin{proof}
Recall that the points in $T(R,B)$ have the form 
		$
		x=\sum_{i=1}^{\infty}R^{-i}b_i 
		$
with $b_i \in B$, 
and the map $\mathcal{R}$ is
		\[
		\mathcal{R}\left(\sum_{i=1}^{\infty}R^{-i}b_{i}\right)=\sum_{i=1}^{\infty}R^{-i}b_{i+1}.
		\]
Denote $\mathcal{R}^k=\mathcal{R} \cdots \mathcal{R}$ and $y=\sum_{i=1}^{\infty}R^{-i}c_i$, we see that 
	\[
	(\mathcal{R}^k x-\mathcal{R}^k y)- R^k(x-y)= -\sum_{i=1}^kR^{k-i}(b_i-c_i)\in \mathbb{Z}.
	\]
Since $m_\tau(x)$ is $\mathbb Z$-periodic, we have 
$$m_\tau(\mathcal{R}^kx-\mathcal{R}^ky)=m_\tau(R^k(x-y))$$ 
for all $x\in T(R,B)$ and $k\in N$.
By Proposition \ref{pr:D}, we obtain
\begin{align*}
D_n(x-(I-R^{-1})^{-1}b)&
=\prod_{k=0}^{n}m_\tau(R^k(x-(I-R^{-1})^{-1}b))\\
&=\prod_{k=0}^{n}m_\tau(\mathcal{R}^kx-
\mathcal{R}^k((I-R^{-1})^{-1}b))\\
&=\prod_{k=0}^{n}m_\tau(\mathcal{R}^kx-\mathcal{R}^k(\sum\limits_{i=0}^{\infty}R^{-i}b))\\
&=\prod_{k=0}^{n}m_\tau(\mathcal{R}^kx-
(I-R^{-1})^{-1}b).
\end{align*}
Combining with Corollary \ref{co:er} and Birkhoff's ergodic theorem, one has that for $\mu$-a.e. $x$ in $T(R,B)$,
\begin{equation*}\label{8}
	\lim_{n \to \infty}\frac{1}{n}\sum_{k=0}^{n-1}\log|m_\tau(\mathcal{R}^kx-(I-R^{-1})^{-1}b)|= \log\Delta(m_{\tau,b}),
\end{equation*}
i.e.,
   \[
     \lim_{n \to \infty}\frac{1}{n}\log D_{n-1}(x-(I-R^{-1})^{-1}b)= \log\Delta(m_{\tau,b}).
	\]
Thus we can get a subset $A\subset T(R,B)$ with measure $\mu(A)> \frac{1}{2}$ such that the limit above is uniform on $A$. 
If $\Delta(m_{\tau,b})>1$ for some $b\in B$, for $x\in A$, taking $1<\rho<\Delta(m_{\tau,b})$, there exists $n_{\rho}$ such that for $n > n_{\rho}$,
\[
		\frac{1}{n}\log D_{n-1}(x-(I-R^{-1})^{-1}b)> \log\rho.
	\]
 
 For $x\in A$, it is easy to see
	\[
\begin{aligned}
			\sup_{n}|D_{n-1}(x-(I-R^{-1})^{-1}b)|&\geq\sup_{n>n_{\rho}}|D_{n-1}(x-(I-R^{-1})^{-1}b)|\\
&\geq\sup_{n>n_{\rho}}\rho^n = +\infty.
		\end{aligned}
		\]
Hence for $\alpha \geq 0,$ the Mock Dirichlet summation operator $S_n (\delta_{(I-R^{-1})^{-1}b})(x)$ satisfies
		\[
		\mu(\{x\in T(R,B):\sup_{n}|S_n(\delta_{(I-R^{-1})^{-1}b})(x)|> \alpha\})\geq \mu(A)\geq\frac{1}{2}.
		\]
By Theorem \ref{thm1.1}, the proof is complete.  
	\end{proof}
 
As an example of Lemma \ref{Classification}, we can now prove Corollary \ref{cor1.2}.

{\bf Proof of Corollary \ref{cor1.2}.}  By \cite[Theorem 1.6]{Yu07}, the quarter Cantor measure $\mu_{4,\{-1,1\}}$ is doubling spectral measure on $T(4,\{-\frac{1}{4},\frac{1}{4}\})$. Under a similarity  transformation,  $\Delta_{2m_{L}}$ in \cite[Example 2.5]{DHS14} is equal to $	\Delta(m_{17,{-1}})$. Hence, according to numerical results in \cite[Example 2.5]{DHS14} and Lemma \ref{Classification}, Corollary \ref{cor1.2} is proved. $\hfill\qed$


\begin{thebibliography}{999}

	
\bibitem[AH14]{AH14} Li-Xiang An, Xing-Gang He, A class of spectral Moran measures. J. Funct. Anal. 266 (1) (2014),
	343–354

\bibitem[Ca09]{Ca09} Marilina Carena, Weak type (1,1) of maximal operators on metric measure spaces. 
Rev. Un. Mat. Argentina 50 (2009), no. 1, 145-159.

\bibitem[Ch90]{Ch90}Michael Christ, A 
$T(b)$
theorem with remarks on analytic capacity and the Cauchy integral. Colloq. Math. 60/61 (1990), no. 2, 601-628.

\bibitem[DHL14]{DHL14} Xin-Rong Dai, Xing-Gang He, Ka-Sing Lau, On spectral N-Bernoulli measures. Adv. Math. 259 (2014), 511–531.


\bibitem[DHL19]{DHL19}Dorin Ervin Dutkay, John Haussermann, Chun-Kit Lai, Hadamard triples generate self-affine spectral measures. Trans. Amer. Math. Soc. 371 (2019), no. 2, 1439–1481.

\bibitem[DHS14]{DHS14}  Dorin Ervin Dutkay, De-Guang Han, Qi-Yu Sun, Divergence of the Mock and scrambled Fourier series on fractal measures.  Trans. Amer. Math. Soc. 366 (2014), no. 4, 2191-2208.

\bibitem[DLW17]{DLW17}  Dorin Ervin Dutkay, Chun-Kit Lai, Yang Wang, Fourier bases and Fourier frames on self-affine measures. Recent developments in fractals and related fields, 87-111, Trends Math., Birkhäuser/Springer, Cham, 2017.

\bibitem[Gu81]{Gu81} Miguel de Guzm\'{a}n, Real Variable Methods in Fourier Analysis, North-Holland Math. Stud. 46, North-Holland, Amsterdam, 1981.

\bibitem[Hu81]{Hu81}John Edward Hutchinson, Fractals and self-similarity. Indiana Univ. Math. J. 30 (1981), no. 5, 713–747.

\bibitem[JP98]{JP98}  Palle Jorgensen, Steen Pedersen, Dense analytic subspaces in fractal $L^2$-spaces. J. Anal. Math. 75 (1998), 185-228.

\bibitem[LW02]{LW02} Izabella {\L}aba, Yang Wang, On spectral Cantor measures. J. Funct. Anal. 193 (2002), no.2, 409–420.


\bibitem[Str06]{Str06}  Robert Strichartz, Convergence of Mock Fourier series. J. Anal. Math. 99 (2006), 333-353.

\bibitem[Yu07]{Yu07}  Po-Lam Yung,
Doubling properties of self-similar measures. 
Indiana Univ. Math. J. 56 (2007), no. 2, 965–990. 

\bibitem[Zy68]{Zy68} Antoni Szczepan Zygmund,  Trigonometric series: Vols. I, II. Second edition, reprinted with corrections and some additions Cambridge University Press, London-New York 1968. 

\end{thebibliography}
\end{document}